\documentclass[12pt, reqno]{amsart}
\usepackage{amsmath, amsthm, amscd, amsfonts, amssymb, graphicx, color}
\usepackage[bookmarksnumbered, colorlinks, plainpages]{hyperref}

\begin{document}
\thispagestyle{empty}
 \setcounter{page}{1}

\begin{center}
{\large\bf  Approximate Generalized Additive-Quadratic Functional Equations on Ternary Banach Algebras} \vskip.30in

{\bf  Sedigheh Jahedi$^{1,*}$ and Vahid Keshavarz$^1$} \\[2mm]
{\footnotesize  $^1$Department of Mathematics, Shiraz University of Technology\\
P. O. Box 71557-13876, Shiraz, Iran}\\
{\footnotesize   jahedi@sutech.ac.ir and v.keshvarz68@yahoo.com}\\
\end{center}\vskip 2mm

 \noindent{\footnotesize{\bf Abstract.}
In this paper, we introduce the concept of j-hom-derivation, $j\in\{1,2\}$ and solve the new generalized additive-quadratic functional equations in the sense of ternary Banach algebras. Moreover, using the fixed point method, we prove its Hyers-Ulam stability.
\footnotetext{*Corresponding author}
\baselineskip=16pt

\theoremstyle{definition}
  \newtheorem{df}{Definition}[section]
    \newtheorem{rk}[df]{Remark}
\theoremstyle{plain}
  \newtheorem{prop}[df]{Proposition}
  \newtheorem{ex}[df]{Example}
  \newtheorem{thm}[df]{Theorem}
  \newtheorem{cor}[df]{Corollary}
 \newtheorem{lem}[df]{Lemma}
 \setcounter{section}{0}
 \numberwithin{equation}{section}
\vskip .1in
\section {\bf{Introduction and preliminaries }}
A ternary Banach algebra $X$ with $\|.\|$ is a complex Banach algebra equipped with a ternary product  $(xyz)\rightarrow xyz $ of $X^{3}$ into $X$. This product is $\mathbb{C}$-linear in the outer variable, conjugate $\mathbb{C}-$linear in the middle variable associative in the sense that
$ xy(zwv)= x(zyw)v=(xyz)wv$ and satisfies $\|xyz\|\leq \|x\|.\|y\|.\|z\|,~\|xxx\|=\|x\|^{3}$ (see \cite{H.Z}). Ternary structures and their extensions, known as n-ary algebras have many applications in mathematical physics and photonics, such as the quark model and Nambu mechanics \cite{k, n}. Today, many physical systems can be modeled as a linear system. The principle of additivity has various applications in physics especially in calculating the internal energy in thermodynamic and also the meaning of the superposition principle. 
   
The Hyers-Ulam stability problem which arises from Ulam's question says that for two given fixed functions $\varphi$ and $\psi,$ the functional equation $\mathcal{E}_1(F)=\mathcal{E}_2(F)$ is stable if for a function $f$ for which $d(\mathcal{E}_1(f),\mathcal{E}_2(f))\leq\varphi$ holds, there is a function $g$ such that $\mathcal{E}_1(g)=\mathcal{E}_2(g)$ and $d(f,g)\leq\psi.$ \cite{g, h, r, u}. In 1941 \cite{h}, Hyers solved the approximately additive mappings on the setting of Banach spaces.

The functional equation $f(x+y)=f(x)+f(y)$ is an additive equation and its solution is called an additive mapping.

\begin{thm}\cite{h}
 Let X and Y be Banach spaces. Assume that $f:X\to Y$ satisfies
 $$\|f(x+y)-f(x)-f(y)\|\leq\varepsilon$$
 for all $x,y,z\in X$ and some $\varepsilon\geq0.$ Then there exists a unique additive mapping $T:X\to Y$ such that
 $$\|f(x)-T(x)\|\leq\varepsilon$$
 for all $x\in X.$
\end{thm}
 First, T. Rassias \cite{r} and Aoki \cite{a} and then a number of authors extended this result by considering the unbounded Cauchy differences in different spaces. For example see \cite{ 22, 20, 1,2,3, jk, kj1, 21}. F. Skof in 1983 \cite{skof}, proved the stability problem of quadratic functional equation between normed and Banach spaces. The functional equation $f(x+y)+f(x-y)=2f(x)+2f(y)$ is called quadratic equation. Cholewa \cite{ch} showed that the Skof's theorem is also true for the mappings defined on abelian groups. Later, a lot of research appeared with various generalization of quadratic functional equations \cite{c, 3.5, 4.5}.\\
Consider the generalized additive-quadratic functional equation
\begin{equation}\label{v1}
\begin{split}
3^jf(\frac{x+y+z}{3})&+f(x)+f(y)+(-1)^jf(z)-2^jf(\frac{x+y}{2})-2^jf(\frac{y+z}{2})-(-1)^j2^jf(\frac{x+z}{2})\\
&=\rho[jf(x+y+z)+jf(x)-f(x+y)-f(x+z)-(j-1)f(y+z)]
\end{split}
\end{equation}
where $\rho\neq0,\pm1$ is a complex number and $j\in\{1,2\}$. In this paper,  we solve \eqref{v1} and show that for $j=1,$ a function which satisfies \eqref{v1} is additive and for $j=2,$ it is quadratic. We also prove its Hyers-Ulam stability by using the fixed point method. To do this, we use the Diaz-Margolis fixed point theorem \cite{D}.


\begin{thm}\label{t1}\cite{D}
Let $(X, d)$ be a complete generalized metric space and let $F: X \rightarrow X$ be a strictly contractive mapping
with Lipschitz constant $0<L<1$. Then for each given element $x\in X$, either
$$d(F^n (x), F^{n+1} (x)) = \infty$$
for all nonnegative integers $n$ or there exists a positive integer $n_0$ such that

$(1)$ $d(F^n (x), F^{n+1}(x)) <\infty, \qquad \forall n\ge n_0$;

$(2)$ the sequence $\{F^n (x)\}$ converges to a unique fixed point $y^*$ of $F$ in the set $Y = \{y\in X \mid d(F^{n_0} x, y) <\infty\}$;

$(4)$ $d(y, y^*) \le \frac{1}{1-L} d(y, F(y))$ for all $y \in Y$.
\end{thm}

\section{Main Results}


Extending the concepts of m-homomorphism, $1\leq m\leq4,$ and hom-derivation which has been introduced by Eshaghi Gordji et. al \cite{e} and park et. al \cite{phom}, respectively, we have the following definitions in the case of ternary Banach algebras. Throughout the paper, X is a ternary Banach algebra.

 We say that a functional  $f$ is a j-mapping, $j\in\{1,2\}$ if for $j=1$ it is an additive and for $j=2,$ f is a quadratic mapping.
\begin{df}\label{1}
A mapping $f:X\to X$ is a ternary j-homomorphism, $j\in\{1,2\},$ if $f$ is a j-mapping and
$$f(xyz)=f(x)f(y)f(z)$$
for all $x,y,z\in X.$
\end{df}
\begin{df}\label{2}
Let $h:X\to X$ be a ternary j-homomorphism, $j\in \{1,2\}.$ A mapping $D:X\to X$ is called a ternary j-hom-derivation if $D$ is a j-mapping and satisfies
$$D(xyz)=D(x)(h(y))^j(h(z))^j+ (h(x))^jD(y)(h(z))^j+(h(x))^j(h(y))^jD(z)$$
for all $x,y,z\in X$.\\

For $j\in\{1,2\}$ and a given mapping $f:X\to X$ consider the equation
\end{df}
\begin{equation}\label{v2}
\begin{split}
\mathcal{E}^jf(x,y,z)=3^jf(\frac{x+y+z}{3})&+f(x)+f(y)+(-1)^jf(z)-2^jf(\frac{x+y}{2})-2^jf(\frac{y+z}{2})-(-1)^j2^jf(\frac{x+z}{2})\\
&-\rho[jf(x+y+z)+jf(x)-f(x+y)-f(x+z)-(j-1)f(y+z)].
\end{split}
\end{equation}

In the following we give the solution of the functional equation \eqref{v1}.
\begin{prop}\label{p1}
Let $f:X\to X$ be a mapping satisfies  $\mathcal{E}^j(f(x,y,z))=0$, $j\in \{1,2\}.$ Then $f(0)=0$ and\\
$(i)~$ for $j=1$, f is an odd and additive mapping;\\
$(ii)~$ for $j=2$, f is a quadratic mapping if f is an even function.
\end{prop}
\begin{proof}
First of all, note that $\mathcal{E}^jf(0,0,0)=0$  implies that $f(0)=0.$\\
$(i)~$ Suppose $j=1$. Let $x=0$ and $y=0$ in \eqref{v2}.Then
\begin{equation}\label{v3}
3f(\frac{z}{3})=f(z)
\end{equation}
Put $x=0$ and replacing $z$ by $y$ in \eqref{v2} and using \eqref{v3}, we obtain $f(2y)-2f(y)=\rho[f(2y)-2f(y)].$ Since $\rho\neq0,\pm1,$ then
\begin{equation}\label{v4}
  f(2y)=2f(y).
\end{equation}
Again put $x=0$ in \eqref{v2}, use \eqref{v3} and \eqref{v4}  we get $\rho[f(y+z)-f(y)-f(z)]=0.$ So
\begin{equation*}
f(y+z)=f(y)+f(z)
\end{equation*}
i.e., $f$ is additive. To show that $f$ is an odd function put $z=-y$ in the above equation, so $f(-y)=-f(y)$.\\
$(ii)~$ Suppose $j=2$ and $f$ is an even function satisfies $\mathcal{E}^2f(x,y,z)=0.$ So from $\mathcal{E}^2f(x,y,-x)=0$ we have 
 \begin{align}\label{v6}
9f(\frac{y}{3})+f(y)+2f(x)-4f(\frac{x+y}{2})-4f(\frac{y-x}{2})=\rho[2f(y)+2f(x)-f(x+y)-f(y-x)].
\end{align}
Letting  $y=0$ and $x=2x$ in \eqref{v6}, then
\begin{equation}\label{v7}
f(2x)=4f(x).
\end{equation}
Again in \eqref{v6} put $x=0$ and use \eqref{v7}. So
\begin{equation}\label{v8}
9f(\frac{y}{3})=f(y).
\end{equation}
Finally, by applying \eqref{v7} and \eqref{v8} in $\mathcal{E}^2f(x,y,z)=0,$
$$f(x+y)+f(x-y)=2f(x)+2f(y).$$
This completes the proof.
\end{proof}
In the sequel, assume $j\in \{1,2\}$ and $\delta,\sigma: X^3\to [0, \infty)$ are two functions satisfy conditions
\begin{equation}\label{v9}
\delta(\frac{x}{2},\frac{y}{2},\frac{z}{2})\leq\frac{k}{2^{j}}\delta(x,y,z),
\end{equation}
\begin{equation}\label{v99}
\sigma(\frac{x}{2},\frac{y}{2},\frac{z}{2})\leq\frac{k}{2^{3j}}\sigma(x,y,z)
\end{equation}
for all $x,y,z\in X$ and some $0<k<1.$ Clearly, $\delta(0,0,0)=\sigma(0,0,0)=0$ and  by induction one can obtain that
\begin{equation*}
2^{nj}\delta(\frac{x}{2},\frac{y}{2},\frac{z}{2})\leq k^n\delta(x,y,z),
\end{equation*}
\begin{equation*}
{2^{3nj}}\sigma(\frac{x}{2},\frac{y}{2},\frac{z}{2})\leq k^n\sigma(x,y,z)
\end{equation*}
for all $n\in \mathbb{N}.$ If $f:X\to X$ is a function such that $\|\mathcal{E}^jf(x,y,z)\|\leq\delta(x,y,z)$
then we have $f(0)=0.$\\

To prove the following results, we need consider two cases. For $j=1$, suppose that $f_j$ and $g_j$ are odd and in case $j=2$ assume $f_j$ and $g_j$ are even functions.
\begin{thm}\label{t2}
Let $j\in\{1,2\}$ and $f_j:X\to X$ be a function satisfies
\begin{equation}\label{v11}
\|\mathcal{E}^jf_j(x,y,z)\|\leq\delta(x,y,z)
\end{equation}
where $\delta:X^3\to [0,\infty)$ fulfill \eqref{v9}. Then there exists a unique ternary j-mapping $h_j:X\to X$ such that
\begin{equation}\label{v111}
\|f_1(x)-h_1(x)\|\leq\frac{k}{2(1-k)}\delta(0,x,-x),
\end{equation}
\begin{equation}\label{v1111}
\|f_2(x)-h_2(x)\|\leq\frac{k}{2(1-k)}\delta(x,0,-x)
\end{equation}
\end{thm}
\begin{proof}
Let $\mathfrak{A}$ be the set of all functions $g:X\to X$ with $g(0)=0$. Define the mapping $Q_j:\mathfrak{A}\to\mathfrak{A}$ by $Q_j(g)=2^{j}g(\frac{x}{2})$  and for every  $ g, h\in \mathfrak{A}$ and $x\in X$ define
$$d_1(g,h)=\inf\{\alpha>0~:~\|g(x)-h(x)\|\leq \alpha\delta(0,x,-x)\},$$
$$d_2(g,h)=\inf\{\alpha>0~:~\|g(x)-h(x)\|\leq \alpha\delta(x,0,-x)\}$$
where $\inf\emptyset=+\infty.$ It is easy to show  that for each $j\in\{1,2\}$,  $d_j$ is a generalized metric on $\mathfrak{A}$ and  $(\mathfrak{A},d_j)$ is a complete generalized metric space.
Let $g,h\in \mathfrak{A}.$  Then
\begin{align*}
  \|Q_1(g)(x)-Q_1(h)(x)\|\leq 2d_{1}(g,h)\frac{k}{2}\delta(0,x,-x),
\end{align*}
\begin{align*}
  \|Q_2(g)(x)-Q_2(h)(x)\|\leq 2^2d_{2}(g,h)\frac{k}{2^2}\delta(x,0,-x)
\end{align*}
Then $d_{j}\Big(Q_j(g),Q_j(h)\Big)\leq k d_{j}(g,h),$ i.e. $Q_j$ is a contraction mapping. In case $j=1,$ put $x=0$  and $z=-y$ in \eqref{v11},  then
$$\|f_1(y)-2f_1(\frac{y}{2})\|\leq\frac{1}{2}\delta(0,y,-y).$$
In case $j=2,$ by setting $y=0$ and $z=-x$ in \eqref{v11} we have
$$\|f_2(x)-2^2f_2(\frac{x}{2})\|\leq\frac{1}{2}\delta(x,0,-x).$$
Above relations imply that $d_j(f_j,Q_j(f_j))\leq\frac{1}{2}.$ Hence by Theorem \ref{t1}, there exist a positive integer $n_0$ and a unique fixed point $h_j$ of $Q_j$ in set $\Omega=\{g\in\mathfrak{A}:~d(Q_j^{n_0}(f),g)<\infty\}$ and $\lim_{n\to\infty}Q_j^n(f_j)(x)=h_j(x),$ . SO,  for all $x\in X,$ $Q_j(h_j)(x)=h_j(x)$ and $\lim_{n\to\infty}{2^{jn}}f(\frac{x}{2^n})=h_j(x).$
Also we have $d_{j}(f_j,h_j)\leq\frac{k}{2(1-k)}.$ This implies that
$$\|f_1(x)-h_1(x)\|\leq\frac{k}{2(1-k)}\delta(0,x,-x),$$
$$\|f_1(x)-h_1(x)\|\leq\frac{k}{2(1-k)}\delta(x,0,-x).$$
But for each $j\in \{1,2\},$
\begin{align*}
\|\mathcal{E}^jh_{j}(x,y,z)\|&=\lim_{n\to\infty}{2^{jn}}\|\mathcal{E}^jf(\frac{x}{2^{n}},\frac{y}{2^{n}},\frac{z}{2^{n}})\|\\
&\leq\lim_{n\to\infty}{2^{jn}}\delta(\frac{x}{2^{n}},\frac{y}{2^{n}},\frac{z}{2^{n}})\\
&\leq\lim_{n\to\infty}k^n\delta(x,y,z)\\
&=0.
\end{align*}
Hence by Proposition \ref{p1}, for each $j\in\{1,2\},h_j$ is a j-mapping and the proof is complete.
\end{proof}
Now, we are going to prove Hyers-Ulam stability of ternary j-hom-derivations in ternary Banach algebras corresponding to the functional equation \eqref{v1}, by using the fixed point method.
\begin{thm}\label{t3}
Assume $\delta,\sigma:X^3\to[0,\infty)$ are two functions which satisfy conditions \eqref{v9} and \eqref{v99} for some constant $k\in (0,1).$ Suppose that $f_j$ and $g_j$ fulfill the following relations
\begin{equation}\label{v13}
\|\mathcal{E}^jf(x,y,z)\|\leq\delta(x,y,z),
\end{equation}
\begin{equation}\label{v14}
\|\mathcal{E}^jg(x,y,z)\|\leq\delta(x,y,z),
\end{equation}
\begin{equation}\label{v15}
\|f_j(xyz)-f_j(x)f_j(y)f_j(z)\|\leq\sigma(x,y,z),
\end{equation}
\begin{equation}\label{v16}
\|g_j(xyz)-g_j(x)(f_j(y))^j(f_j(z))^j-(f_j(x))^jg_j(y)(f_j(z))^j-(f_j(x))^j(f_j(y))^jg_j(z)\|\leq\sigma(x,y,z).
\end{equation}
for all $x,y,z\in X.$ Then there exist a unique ternary j-homomorphism $H_j:X\to X$ and a unique ternary j-hom-derivation $D_j:X\to X$ such that
\begin{equation}\label{v17}
\|f_1(x)-H_1(x)\|\leq\frac{k}{2(1-k)}~\delta(0,x,-x),
\end{equation}
\begin{equation}\label{v177}
\|f_2(x)-H_2(x)\|\leq\frac{k}{2(1-k)}~\delta(x,0,-x),
\end{equation}
\begin{equation}\label{v18}
\|g_1(x)-D_1(x)\|\leq\frac{k}{2(1-k)}~\delta(0,x,-x),
\end{equation}
\begin{equation}\label{v188}
\|g_2(x)-D_2(x)\|\leq\frac{k}{2(1-k)}~\delta(x,0,-x)
\end{equation}
for all $x\in X$ and some  $0<k<1.$
\end{thm}
\begin{proof}
Let $\mathfrak{A},d_j$ and $Q_j$, $j\in \{1,2\},$ be those as defined in the proof of Theorem \ref{t2}. Similar to the proof of Theorem \ref{t2}, there exist unique j-mappings $H_j, D_j$ from X into X such that
\begin{equation}\label{v19}
H_j(x)=\lim_{n\to\infty}Q_j^n(f_j)(x)=\lim_{n\to \infty}{2^{jn}}f_j(\frac{x}{2^n}),
\end{equation}
\begin{equation}\label{v20}
D_j(x)=\lim_{n\to\infty}Q_j^n(g_j)(x)=\lim_{n\to \infty}{2^{jn}}g_j(\frac{x}{2^n})
\end{equation}
and satisfying \eqref{v17}, \eqref{v177},\eqref{v18} and\eqref{v188} as desired. Mappings $H_j$, $j=\{1,2\},$ are ternary j-homomorphism. In fact, by \eqref{v15} and \eqref{v19}, we have
\begin{align*}
  \|H_j(xyz)-H_j(x)H_j(y)H_j(z)\|&=\lim_{n\to\infty}{2^{3nj}}\|f_j(\frac{xyz}{2^{3n}})-f_j(\frac{x}{2^n})f_j(\frac{y}{2^n})f_j(\frac{z}{2^n})\|\\
  &\leq \lim_{n\to\infty}2^{3nj}\sigma(\frac{x}{2^n},\frac{y}{2^n},\frac{z}{2^n})\\
  &\leq \lim_{n\to\infty}k^n\sigma(x,y,z)\\
  &=0.
\end{align*}
Using \eqref{v16} and \eqref{v20} shows that $D_j$ is a ternary j-hom-derivation. In fact,
\begin{align*}
\|D_j&(xyz)-D_j(x)(H_j(y))^j(H_j(z))^j-(H_j(x))^jD_j(y)(H_j(z))^j-(H_j(x))^j(H_j(y))^jD_j(z)\|\\
&=\lim_{n\to\infty}{2^{3nj}}\|g_j(\frac{xyz}{2^{3n}})-g_j(\frac{x}{2^n})(f_j(\frac{y}{2^n}))^j(f_j(\frac{z}{2^n}))^j-(f_j(\frac{x}{2^n}))^jg_j(\frac{y}{2^n})(f_j(\frac{z}{2^n}))^j-(f_j(\frac{x}{2^n}))^j(f_j(\frac{y}{2^n}))^jg_j(\frac{z}{2^n})\|\\
&\leq \lim_{n\to\infty}{2^{3nj}}\sigma(\frac{x}{2^n},\frac{y}{2^n},\frac{z}{2^n})\\
&\leq\lim_{n\to\infty}k^n\sigma(x,y,z)\\
&=0.
\end{align*}
Now, the proof is complete.
\end{proof}
In Theorem \ref{t2} and Theorem \ref{t3}, by taking  $k=2^{1-r}$ and
$$\delta(x,y,z)=\sigma(x,y,z)=s(\|x\|^r+\|y\|^r+\|z\|^r)$$
where $x,y,z\in X,~r\neq1$ and s are nonnegative real numbers, we obtain the following result.
\begin{cor}\label{t4}
Let $r\neq1$ and s be two elements of $\mathbb{R}_+.$ Suppose that $\delta(x,y,z)=\sigma(x,y,z)=s(\|x\|^r+\|y\|^r+\|z\|^r)$. Assume $f_j,g_j:X\to X,~j\in \{1,2\}$ are functions satisfying in \eqref{v13}, \eqref{v14}, \eqref{v15} and \eqref{v16}. Then there exist ternary j-homomorphism $H_j$ and ternary j-hom-derivation $D_j$ such that
\begin{equation*}
\left\{
  \begin{array}{ll}
    \|f_j(x)-H_j(x)\|\leq\frac{2s}{2-2^r}\|x\|^r& \\
    ~~~~~& \hbox{for~ $r>1$} \\
    \|g_j(x)-D_j(x)\|\leq\frac{2s}{2-2^r}\|x\|^r&
  \end{array}
\right.
\end{equation*}
and
\begin{equation*}
\left\{
  \begin{array}{ll}
    \|f_j(x)-H_j(x)\|\leq\frac{2s}{2^r-2}\|x\|^r& \\
    ~~~~~& \hbox{for~ $r<1$} \\
    \|g_j(x)-D_j(x)\|\leq\frac{2s}{2^r-2}\|x\|^r&
  \end{array}
\right.
\end{equation*}
\end{cor}
{\bf Declarations of interest}: none.

\end{document}